\documentclass{amsart}

\usepackage{amsthm}

\usepackage{mathtools}

\DeclarePairedDelimiter\floor{\lfloor}{\rfloor}

\newtheorem{theorem}{Theorem}
\newtheorem{lemma}[theorem]{Lemma}
\newtheorem{proposition}[theorem]{Proposition}
\newtheorem{corollary}[theorem]{Corollary}

\newtheorem{remark}[theorem]{Remark}

\usepackage{graphicx}
\usepackage{amssymb}
\usepackage{epstopdf}
\usepackage{hyperref}
\usepackage{xypic}

\newcommand\CC{\mathbb{C}}

\newcommand\FF{\mathbb{F}}

\newcommand\PP{\mathbb{P}}
\newcommand\QQ{\mathbb{Q}}
\newcommand\RR{\mathbb{R}}

\newcommand\ZZ{\mathbb{Z}}

\newcommand\cC{\mathcal{C}}

\newcommand\cF{\mathcal{F}}

\newcommand\cI{\mathcal{I}}

\newcommand\cT{\mathcal{T}}

\newcommand\cY{\mathcal{Y}}

\newcommand\Sp{\text{Sp}}

\newcommand\h{\mathfrak{h}}

\newcommand\chat{\widetilde{\mathcal{C}}}

\newcommand\C{\cC}
\newcommand\Thetatilde{\widetilde\Theta}

\newcommand\chatx{\widetilde{\cC}_X}
\newcommand\chate{\widetilde{\cC}_E}

\title{The Topology of the Zero Locus of A Genus 2 Theta Function}
\author{Kevin Kordek}
\date{}
\begin{document}
\maketitle
\begin{abstract}
Mess showed that the genus 2 Torelli group $T_2$ is isomorphic to a free group of countably infinite rank by showing that genus 2 Torelli space is homotopy equivalent to an infinite wedge of circles. As an application of his computation, we compute the homotopy type of the zero locus of any classical genus 2 theta function in $\h_2 \times \CC^2$, where $\h_2$ denotes rank 2 Siegel space. Specifically, we show that  the zero locus of any such function is homotopy equivalent to an infinite wedge of 2-spheres.
\end{abstract}
\section{Introduction}
Theta functions are classical holomorphic functions of several variables defined on $\h_g\times \CC^g$, where $\h_g$ denotes the rank $g$ Siegel upper half-space. Since the $19$th century, theta functions have played a central role in algebraic geometry, especially the geometry of curves and abelian varieties. 

Although these functions have been the object of intense study for well over a century, it appears that essentially nothing is known about the specific topological properties of their zero loci once $g\geq 2$. In this paper, we will use Geoffrey Mess's explicit computation \cite{mess1992torelli} of the isomorphism type of the genus 2 Torelli group to prove the following result.
\begin{theorem}\label{theorem1}
Let $\vartheta_{\alpha}: \h_2\times \CC^2\rightarrow \CC$ denote a theta function of genus 2. Then the zero locus of $\vartheta_{\alpha}$ is homotopy equivalent to an infinite bouquet of 2-spheres.
\end{theorem}

The outline of the paper is as follows. In Section 2, we present the necessary background in mapping class groups and the geometry of algebraic curves and their jacobians. 

The Torelli space $\cT_2^c$ is the moduli space of genus 2 curves of compact type with an integral homology framing (see Section \ref{subsection2.2}). In Section 3, we establish a link between zero loci of genus 2 theta functions and the universal family of curves $\C\rightarrow \cT_2^c$, which we now describe. 

The rank 2 Siegel space $\h_2$ may be regarded as the moduli space of principally polarized abelian surfaces with a homology framing (see Section \ref{subsection2.3}). It also carries a universal family $\mathfrak{X}\rightarrow \h_2$ of principally polarized abelian surfaces. Let $\Theta_{\alpha}$ denote the zero locus of $\vartheta_{\alpha}$ on $\mathfrak{X}$.

 \begin{proposition}\label{proposition8}
As a family of framed curves, the family $f: \Theta_{\alpha}\rightarrow \h_2$ is isomorphic to the universal family $\C\rightarrow \cT_2^c$.
\end{proposition}
We will also establish the geometric facts about the universal family needed to compute the integral homology groups of the zero locus of $\vartheta_{\alpha}$ in $\h_2\times \CC^2$. It will turn out that this information suffices to determine the homotopy type since, as we will show, $\Theta_{\alpha}$ is simply connected. Proposition \ref{proposition8} is of crucial importance here, as it allows us to compute many of these homology groups in terms of the homology of the genus $2$ Torelli group $T_2$. 

Sections 4-6 are technical in nature, and this is where the homological computations will take place. The key results here are Propositions \ref{prop22}, \ref{prop31} and \ref{prop34} which together, along with the homological form of Whitehead's Theorem, imply that $\Theta_{\alpha}$ is homotopy equivalent to an infinite bouquet of 2-spheres. 

The paper concludes with Section 7, where we use commutator calculus in the fundamental group of a genus 2 surface to construct an infinite family of homotopically non-trivial maps $S^2\rightarrow \Theta_{\alpha}$. 

\subsection{Acknowledgements}
The main results of this paper were a part of my thesis. I would like to thank my advisor Dick Hain for suggesting this problem. I thank the referee for numerous comments and suggestions which helped to significantly improve the quality of this paper. 
\section{Background Material}

\subsection{Mapping Class Groups}

Fix a closed orientable surface of genus $g$ with $r$ marked points $S_{g,r}$. The mapping class group $\Gamma_{g,r}$ is the group of isotopy classes of orientation-preserving diffeomorphisms of $S_{g,r}$ that fix the marked points. When $r = 0$, it will be omitted from the notation.

Note that the isomorphism type of the group $\Gamma(S_{g,r})$ depends only on $g$ and $r$. The \emph{Torelli group} $T_g$ is the kernel of the natural map $\Gamma_g\rightarrow \Sp_g(\ZZ)$. In other words, the elements of $T_g$ are isotopy classes of orientation preserving diffeomorphisms of $S_g$ which act trivially on $H_1(S_g,\ZZ)$.

The Torelli group $T_g$ is not well-understood when $g\geq 3$. Mess \cite{mess1992torelli} was, however, able to give a very explicit description of $T_2$. To describe his result, we will need to introduce some terminology. 

Let $(V,\omega)$ be a free abelian group endowed with a symplectic form. A \emph{symplectic splitting} of $V$ is a direct sum decomposition $V = V_+\oplus V_-$ such that $V_+$ is orthogonal to $V_-$ with respect to $\omega$.  If $(V,\omega)$ is equal to $H_1(S_g,\ZZ)$ equipped with 
the 
intersection form, such a decomposition is called a \emph{homology splitting}. Let $c$ denote a separating simple closed curve (SSCC) on $S_g$. Then $c$ induces a homology splitting on $S_g$ in the following way. Cutting along $c$ divides $S_g$ into two closed subsurfaces, $S_g'$ and $S_g''$, each with one boundary component. Then $H_1(S_g',\ZZ)\oplus H_1(S_g'',\ZZ)$ gives a splitting of $H_1(S_g,\ZZ)$, since both summands are generated by cycles supported on one of the two sides of $c$. Conversely, every homology splitting on $S_g$ arises in this way \cite[Section 6]{johnsonconjugacy}. Two SSCC's $c_1,c_2$ on $S_g$ induce the same homology splitting if and only if $c_1$ is carried into $c_2$ by an element of $T_g$ (see, for example, \cite[Section 6]{johnsonconjugacy} or \cite[p.186]{farb2011primer}).
\begin{theorem}[Mess, \cite{mess1992torelli}]\label{theorem 9}
The Torelli group $T_2$ is isomorphic to a free group of countably infinite rank. There is precisely one free generator for each homology splitting on $S_2$. The generator corresponding to a homology splitting $\beta$ is a Dehn twist $T_{c_{\beta}}$ on a SSCC $c_{\beta}$ realizing $\beta$.
\end{theorem}

\begin{remark}
\emph{It is important to note that the isotopy class of $c_{\beta}$ cannot be chosen arbitrarily --- it is contained in a \emph{specific} isotopy class of curves which induce the homology splitting $\beta$. Unfortunately, this isotopy class has not been determined explicitly. In \cite{mess1992torelli}, Mess conjectures that the isotopy class of $c_{\beta}$ is represented by shortest-length closed geodesic inducing the homology splitting $\beta$.
}
\end{remark}
\subsection{Moduli of Curves}\label{subsection2.2}
Assume that $g\geq 2$ and let $\mathcal{X}_g$ denote genus $g$ Teichm\"{u}ller space. The mapping class group $\Gamma_g$ acts biholomorphically, properly discontinuously, and virtually freely on $\mathcal{X}_g$. The moduli space of genus $g$ curves $M_g$ is, as an analytic variety, isomorphic to the quotient space $\Gamma_g\backslash \mathcal{X}_g$. The Torelli space $\cT_g$ is the quotient $T_g\backslash \mathcal{X}_g$. It is the moduli space of genus $g$ curves with a symplectic basis for the first integral homology, i.e. a homology framing. Because $T_g$ is torsion-free, $\cT_g$ is a  $K(T_g,1)$-space.

A curve of compact type is a stable nodal curve all of whose irreducible components are smooth and whose dual graph is a tree. The space $\cT_g$ can be enlarged to a complex manifold $\cT_g^c$ whose points parametrize genus $g$ curves of compact type with a homology framing \cite{hain2006finiteness}. The $\Sp_g(\ZZ)$-action on $\cT_g$ extends to $\cT_g^c$, and the moduli space $M_g^c$ of curves of compact type is the quotient $\Sp_g(\ZZ)\backslash \cT_g^c$ by \cite{hain2006finiteness}.
\subsection{Abelian Varieties}\label{subsection2.3}

The material in this section is classical. Most of it can be found in \cite[Ch.2, Sec. 6]{griffiths2014principles}, for example. A \emph{polarized abelian variety} $A$ is an abelian variety along with a cohomology class $\theta\in H^2(A,\ZZ)$  (called a polarization) represented by a  positive, integral, translation-invariant $(1,1)$ form . A \emph{principally polarized} abelian variety (ppav) is a polarized abelian variety whose polarization, viewed as a skew-symmetric bilinear form on $H_1(A,\ZZ)$ can be put into the form
\begin{equation}
\left(\begin{array}{cc}0 & I_{g\times g} \\ -I_{g\times g}& 0\end{array}\right)
\end{equation}
 with respect to some integral basis for $H_1(A,\ZZ)$. A \emph{framing} on a ppav $(A,\theta)$ is a choice of basis for $H_1(A,\ZZ)$ which is symplectic with respect to $\theta$. The symplectic basis will be denoted $\cF$.

Given a framed ppav $(A,\theta, \cF)$ one can find a unique basis of holomorphic 1-forms for which the period matrix is of the form 
\begin{equation}
\left(\begin{array}{cc}\Omega \\ I_{g\times g}\end{array}\right).
\end{equation}
In this case, $\Omega$ is called the \emph{normalized period matrix} of $(A,\theta,\cF)$. Up to isomorphism, this framed ppav can be recovered as the quotient $A_{\Omega} := \CC^g/\Lambda(\Omega)$ where $\Lambda(\Omega) = \ZZ^g + \ZZ^g\Omega\subset \CC^g$ \cite[pp.306-307]{griffiths2014principles}.

The space of normalized period matrices $\Omega$ of $g$-dimensional ppav's is known as \emph{Siegel space of rank $g$}. It will be denoted by $\h_g$. It is 
the open subspace of $M_{g\times g}(\CC)\cong \CC^{g(g+1)/2}$ consisting of symmetric matrices with positive-definite imaginary part \cite[Prop. 8.1.1]{birkenhake2004complex}.
The symplectic group $\Sp_g(\ZZ)$ of integral symplectic $2g\times 2g$ matrices holomorphically on $\h_g$ by the following formula \cite[Prop. 8.2.2]{birkenhake2004complex}:
\begin{equation}
M\cdot \Omega = (A\Omega+B)(C\Omega+D)^{-1}\ \ \ \ \ \text{where} \ M = \left(\begin{array}{cc}A & B \\ C& D \end{array}\right)\in \Sp_g(\ZZ).
\end{equation}

The product of two ppav's $(A_1, \theta_1), (A_2, \theta_2)$ is the ppav whose underlying torus is $A_1\times A_2$ and whose polarization is the product polarization  $p_1^*\theta_1\oplus p_2^*\theta_2$, where $p_j : A_1\times A_2$ is projection onto the $j^{th}$ factor \cite{hain2002rational}. A ppav which can be expressed as a product of two ppav's of positive dimension is called \emph{reducible}. The sublocus of $\h_g$ parametrizing reducible ppav's will be denoted $\h_g^{red}$. By \cite{hain2006finiteness} there is an equality
$$\h_g^{red} = \bigcup_{j=1}^{\floor{g/2}}\bigcup_{\phi\in \Sp_g(\ZZ)} \phi\left(\h_j\times \h_{g-j}\right).$$

\begin{remark}\label{remark11}

The following notation shall be used throughout the paper. If a framed (necessarily reducible) ppav has period matrix 
\begin{equation*}
\Omega = \left(\begin{array}{cc}\Omega_1 & 0 \\ 0& \Omega_2 \end{array}\right)
\end{equation*}
for matrices $\Omega_j \in \h_{g_j}$ with $g_1+g_2 = g$ and $g_j >0$, then we will write $\Omega = \Omega_1\oplus \Omega_2$.

\end{remark} 
\subsection{The Period Map}
There is a holomorphic map $\cT_g\rightarrow \h_g$ which sends the isomorphism class $[C;\cF]$ of a framed curve to its period matrix. It is called the \emph{period map \cite{hain2006finiteness}}. The sharp form of the Torelli Theorem \cite{mess1992torelli} states that the 
period map is a double branched cover of its image, and that it is injective along the locus of hyperelliptic curves. 

The period map extends to a proper holomorphic map $\cT_g^c\rightarrow \h_g$ (the \emph{extended period map}) that sends a framed curve of compact type to its period matrix \cite{hain2006finiteness}. The Torelli Theorem no longer holds in this setting; the fiber over a point in $\h_g^{red}$ will have positive dimension (see, for example, \cite{hain2006finiteness}).

The boundary $\cT_g^{c,red}: = \cT_g^c - \cT_g$ is equal to the preimage in $\cT_g^c$ of $\h_g^{red}$ under the period map \cite{hain2006finiteness}.

\subsection{Theta Functions}
A \emph{theta function of characteristic $\delta$} is a holomorphic function $\vartheta_{\delta}: \h_g\times \CC^g\rightarrow \CC$ defined by the following series:
\begin{equation}\label{thetaseries}
\vartheta_{\delta}(\Omega,z) = \sum_{m\in \ZZ^g}\exp(\pi i\left( (m+\delta')\Omega (m+\delta')^T + 2(m+\delta')(z+\delta'')^T\right))
\end{equation}
where $\delta = (\delta', \delta'') \in \frac{1}{2}\ZZ^{g}/\ZZ^g\oplus \frac{1}{2}\ZZ^{g}/\ZZ^g = \frac{1}{2}\ZZ^{2g}/\ZZ^{2g}\cong \FF^{2g}_{2}$ (see, for example, \cite{grushevsky}). The vector $\delta$ corresponds to a 2-torsion point of $A_{\Omega}$ and $\vartheta_{\delta}$ is called a theta function \emph{of characteristic $\delta$}. 

By differentiating the series (\ref{thetaseries}), one shows that $\vartheta_{\delta}$ satisfies the \emph{heat equation}:
\begin{equation}
2\pi i (1+\delta_{jk})\frac{\partial \vartheta_{\delta}}{\Omega_{jk}} = \frac{\partial ^2 \vartheta_{\delta}}{\partial z_j\partial z_k}
\end{equation}
where $\Omega = (\Omega_{jk})$ and $\delta_{jk}$ is the Kronecker delta \cite{grushevsky}.

The theta function $\vartheta_{\delta}$ is said to be \emph{even} (\emph{odd}) if for fixed $\Omega$ it is even (odd) as a function of $z$. Equivalently, $\vartheta_{\delta}$ is even (odd) if the number $\text{ord}_{z = 0}\vartheta_{\delta}(\Omega,z)$ is even (odd). If $\vartheta_{\delta}$ is even (odd), then the characteristic $\delta$ is also said to be even (odd). Equivalently, $\delta$ is even (odd) if and only if the number $\delta' (\delta'')^T\in \FF_2$ is equal to 0 (or 1) (see \cite{grushevsky}). Direct calculation shows that the theta function $\vartheta_{\delta}(\Omega,z)$ is equal to the product of $\vartheta_0(\Omega, z+ \delta' + \delta''\Omega)$ with a non-vanishing holomorphic function on $\h_2\times \CC^2$. 

If a characteristic $\delta\in \frac{1}{2}\ZZ^{2g}$ can be written in the form
\begin{equation}
\delta = \left((\delta_1',\ldots, \delta_n'), (\delta_1'',\ldots,\delta_n'')\right)
\end{equation}
 for some characteristics $\delta_j = (\delta_j',\delta_j'')\in \frac{1}{2}\ZZ^{2g_j}$ with $g_1+\cdots + g_n = g$, then we write $\delta = \delta_1\oplus\cdots \oplus\delta_n$. If $\Omega$ is a block matrix of the form $\Omega_1\oplus\cdots \oplus \Omega_n$, with $\Omega_j\in\h_{g_j}$, then there are characteristics $\delta_j\in \frac{1}{2}\ZZ^{2g_j}$ such that $\delta = \delta_1\oplus\cdots \oplus \delta_n$ and 
\begin{equation}\label{thetaproduct}
\vartheta_{\delta}(\Omega,z) = \vartheta_{\delta_1\oplus\cdots \oplus\delta_n}(\Omega,z) = \prod_{1\leq j\leq n}\vartheta_{\delta_j}(\Omega_j, z(j))
\end{equation}
 where $z(j) = (z_{g_1+\cdots g_j+1},\ldots , z_{g_1+\cdots + g_{j+1}})$ (see, for example, \cite{grushevskysalvatimanni}).
\\\\\indent 
The restriction of the theta function $\vartheta_{\delta}$ to $z = 0$ is the \emph{thetanull of characteristic $\delta$}. It is a holomorphic function
\begin{equation}
\vartheta_{\delta}(-,0): \h_g\rightarrow \CC.
\end{equation}
The thetanull $\vartheta_{\delta}(-,0)$ is called even or odd, respectively, if the corresponding theta function is. The odd thetanulls vanish identically \cite{grushevsky}.

Theta functions enjoy special transformation properties with respect to the $\Sp_g(\ZZ)$-action on $\h_g\times \CC^g$ given by the formula
$M\cdot (\Omega, z) = \left(M\cdot \Omega,\  z\cdot (C\Omega+D)^{-1}\right)$. More specifically, we have the function $\vartheta_{\delta}(M\cdot \Omega,\ z\cdot (C\Omega+D)^{-1})$ is equal to $u\cdot \vartheta_{\delta'}(\Omega,z)$, where $u$ is a specific (very complicated) non-vanishing holomorphic function depending on $\delta, M, \Omega$ and $z$ and $\delta'$ is a characteristic depending on $\delta$ and $M$ given by another complicated formula. Precise formulas are given can be found in \cite[p.227]{birkenhake2004complex}.
Thetanulls obey similar transformation laws.

By \cite[p.324]{birkenhake2004complex}, for fixed $\Omega_0\in \h_g$, the theta function $\vartheta_{\delta}(\Omega,z)$ defines a section of a holomorphic line bundle in the ppav $A_{\Omega_0}$ whose first Chern class coincides with the polarization.

Let $C$ be a smooth curve of genus $g$, and let $W_{g-1}\subset \text{Pic}^{g-1}(C)$ denote the codimension 1 subvariety consisting of classes of effective divisors on $C$ of degree $g-1$. Choose a theta characteristic $\alpha$ on $C$, i.e. a square root of the canonical bundle of $C$. The map $\text{Pic}^{g-1}(C)\rightarrow \text{Pic}^0(C)$ defined by $x\rightarrow x-\alpha$ sends $W_{g-1}$ to the locus $W_{g-1}-\alpha$ in $\text{Pic}^0(C)$. By Riemann-Roch, the divisor $W_{g-1}-\alpha$ is a symmetric theta divisor \cite[p.324]{birkenhake2004complex}; that is, it is preserved by the involution $x\rightarrow -x$, and its Chern class is equal to the polarization of $\text{Pic}^0(C)$ furnished by the cup product \cite[pp.327-328]{griffiths2014principles}. It will be denoted by $\Theta_{\alpha}$. This is a geometric formulation of \emph{Riemann's Theorem} \cite[Theorem 11.2.4]{birkenhake2004complex}.

The parity of $\Theta_{\alpha}$ matches that of the theta characteristic $\alpha$ \cite[Proposition 11.2.6]{birkenhake2004complex}. Let $\delta$ be a characteristic such that $\Theta_{\alpha}$ is the zero divisor of $\vartheta_{\delta}$. Then \emph{Riemann's Singularity Theorem} \cite[Theorem 11.2.5]{birkenhake2004complex} states that for $L\in \text{Pic}^{g-1}(C)$, 
\begin{equation*}
\text{mult}_{L}(W_{g-1}) = \text{mult}_{L-\alpha}(\Theta_{\alpha}) = \text{ord}_{L-\alpha}(\vartheta_{\delta}(\Omega_0,-)) = h^0(L)
\end{equation*}

\subsection{Genus 2 Curves and their Jacobians}\label{section2.4.2}
 
Let $C$ denote a smooth curve of genus 2. Then there is a special interaction between the theta characteristics, Weierstrass points, and the jacobian of $C$. Since the canonical linear system gives a 2:1 covering $C\rightarrow \PP^1$, the Weierstrass points of $C$ are easily seen, via the Riemann-Roch Theorem, to coincide with the odd theta characteristics on $C$. 

By Abel's Theorem, the choice of a Weierstrass point $w$ on $C$ gives rise to an embedding $C\hookrightarrow \text{Pic}^0(C)$ defined by $p\rightarrow p-w$. The image is an odd theta divisor in $\text{Pic}^0(C)\cong \text{Jac}(C)$ by Riemann's Singularity Theorem.  Thus any theta divisor on $\text{Jac}(C)$ is isomorphic to $C$ and is, in particular, a non-singular curve (see also \cite{mess1992torelli}). The Riemann Singularity Theorem gives a relation
\begin{equation}
\text{ord}_{p-w}(\vartheta_{\delta}(\Omega_0,-)) = h^0(p) = 1
\end{equation}
where $\Omega_0$ is the period matrix of $\text{Jac}(C)$ and $\vartheta_{\delta}$ is the theta function whose zero divisor is $\Theta_{w}$. This implies that any theta function on the jacobian of a smooth curve of genus 2 vanishes only to first order along its zero locus.
\\\\\indent 
If $C$ is a \emph{singular} genus 2 curve of compact type, then, by definition, it must consist of two elliptic curves $E_1,E_2$ joined at a single node. The generalized jacobian of $C$ is the product $E_1\times E_2$ (see, for example, \cite{mess1992torelli}). Any theta divisor on a 1-dimensional ppav consists of a single point, and, from our description of the theta divisor on a reducible ppav, it follows that theta divisor of $E_1\times E_2$ is of the form $E_1\times \{p_2\}\cup \{p_1\}\times E_2$, for some points $p_j\in E_j$ (see \cite{mess1992torelli}). In particular, $C$ is again isomorphic to the theta divisor of its (generalized) jacobian. Combining this with equation (\ref{thetaproduct}) above, we see that any theta function $\vartheta_{\delta}(\Omega_0,z)$ on $E_1\times E_2$ vanishes to first order along its zero locus, except at a single point  where it vanishes to order 2.

\section{The Geometry of the Universal Family $\C$}

Once and for all, define $\vartheta$ to be the theta function $\vartheta_{\delta}$ with characteristic $\delta = (1/2,1/2,1/2, 1/2)$. For the purposes of this paper, we have found it easier to work with $\vartheta$ than other theta functions. 

For another theta function $\vartheta_{\alpha}$, direct calculation shows that there is a non-vanishing holomorphic function $u$ on $\h_2\times \CC^2$ and a characteristic $\epsilon = (\epsilon', \epsilon'')$ such that $\vartheta(\Omega, z) = u\vartheta_{\alpha}(\Omega, z + \epsilon' + \epsilon''\Omega)$. The holomorphic map $\h_2\times \CC^2\rightarrow \h_2\times \CC^2$ defined by $(\Omega,z)\rightarrow (\Omega, z + \epsilon' + \epsilon''\Omega)$ is easily seen to be a biholomorphism; in particular it is a homotopy equivalence. Thus the zero locus of $\vartheta$ is homotopy equivalent to that of $\vartheta_{\alpha}$. 

To prove Theorem \ref{theorem1}, it therefore suffices to show that the zero locus of $\vartheta$ is homotopy equivalent to an infinite bouquet of 2-spheres. Our work on this occupies Sections 3-6 of the paper. 

\vspace{.1in}
Let $\mathfrak{X}_2\rightarrow \h_2$ denote the universal family of framed principally polarized abelian surfaces over $\h_2$.  This can be constructed as the quotient of $\h_2\times \CC^2$ by $\ZZ^4$ via the action
\begin{equation*}
(m_1,m_2)\cdot (\Omega,z) = (\Omega, z + m_1 + m_2\Omega).
\end{equation*}
Since it has a global homology framing, it is a topologically trivial fiber bundle over $\h_2$ with fiber a real 4-torus. Any theta function $\vartheta_{\alpha}:\h_2\times \CC^2\rightarrow \CC$ can be viewed as a section of a holomorphic line bundle on $\mathfrak{X}_2$. The zero locus $\Theta$ of $\vartheta$ in $\mathfrak{X}_2$ can be viewed as a family of curves over $\h_2$.

The geometry of $\Theta\rightarrow \h_2$ can be described explicitly using the special properties of theta functions. This will allow us to give a very concrete description of the geometry of the universal curve $\cC_2\rightarrow \cT_2^c$. We will be able to deduce results about the geometry and topology of the universal covering $\chat\rightarrow \cC_2$, since, as we will show, the universal cover of $\Theta$ is simply the preimage of $\Theta$ in $\h_2\times \CC^2$, i.e. the zero locus of $\vartheta$ in $\h_2\times \CC^2$.

\subsection{The Zero Locus $\Theta$ as a Family of Curves}
Recall that a proper, surjective holomorphic map $\varphi:X\rightarrow S$ between complex analytic spaces is called a \emph{family of nodal curves} if it is flat and its fibers are nodal curves \cite{arbarello2011geometry}. 

Our first goal is to show that $\Theta\rightarrow \h_2$ is a family of curves. This involves checking that the family is flat. There is a convenient test for flatness that applies when the source and target are both connected complex manifolds, namely that the fibers are equidimensional (see, for example, \cite[pp.113-114]{grauert1994several}). In order to apply this criterion, we will first show that $\vartheta$ defines a complex submanifold $\Theta$ of $\mathfrak{X}_2$. 
\begin{lemma}
The zero locus $\Theta\subset \mathfrak{X}_2$ is a non-singular subvariety. 
\end{lemma}
\begin{proof}
This condition can be checked locally. At any point $p\in \Theta\subset \mathfrak{X}_2$, there is a defining function for $\Theta$ of the form $\vartheta_{\alpha}(\Omega,z)$ on an open subset $U$ of $\h_2\times \CC^2$. By the heat equation, the gradient of this holomorphic function is
\begin{equation}
\left(\frac{1}{2\pi i}\frac{\partial^2 \vartheta}{\partial z_1^2 }, \frac{1}{2\pi i}\frac{\partial^2 \vartheta}{\partial z_1\partial z_2 }, \frac{1}{2\pi i}\frac{\partial^2 \vartheta}{\partial z_2^2 }, \frac{\partial \vartheta}{\partial z_1 }, \frac{\partial \vartheta}{\partial z_2 }\right).
\end{equation}
This can never vanish, as that would imply that there is a period matrix $\Omega'$ theta function $\vartheta_{\alpha}(\Omega',-)$ has a zero of multiplicity at least 3. This would contradict the fact that $\vartheta_{\alpha}(\Omega',-)^{-1}(0)$ has at worst nodal singularities, as explained in Section \ref{section2.4.2}.
\end{proof}
\begin{proposition}
The projection $f:\Theta\rightarrow \h_2$ is a family of nodal curves. Every fiber is a genus 2 curve of compact type, and the fiber over $\Omega_0\in \h_2$ is naturally a framed curve with period matrix $\Omega_0$.  
\end{proposition}
\begin{proof}
Since the map $\mathfrak{X}_2\rightarrow \h_2$ is proper and $\Theta$ is a closed analytic subset of $\mathfrak{X}_2$, it follows that $p$ is proper holomorphic map. The fiber $\Theta_{\Omega_0}$ over $\Omega_0\in \h_2$ is isomorphic, as an analytic space, to the intersection $\Theta\cap \mathfrak{X}_{2,\Omega_0}$. This is the analytic subspace of $\mathfrak{X}_2$ locally cut out by functions
\begin{equation}
\Omega_{11}-(\Omega_0)_{11}\ \ \ \ \Omega_{12}-(\Omega_0)_{12}\ \ \ \ \ \Omega_{22}-(\Omega_0)_{22}\ \ \ \ \ \vartheta(\Omega,z).
\end{equation}
From here it follows that the fiber $\Theta_{\Omega_0}$ is isomorphic to the zero locus of $\vartheta(\Omega_0,z)$ inside the complex torus $\CC^2/\Lambda(\Omega_0)$. This is a genus 2 curve of compact type (hence a nodal curve), and it naturally inherits the homology framing on $\mathfrak{X}_{2,\Omega_0}$. With respect to this framing, the curve $\Theta_{\Omega_0}$ has period matrix $\Omega_0$. Since $\Theta$ and $\h_2$ are connected complex manifolds and the fibers of $f$ are equidimensional, if follows that $f$ is flat.
\end{proof}
Since framed curves of compact type have no automorphisms, the Torelli space of compact type curves $\cT_2^c$ is a fine moduli space. As such, there is a universal family of framed curves of compact type $\cC\rightarrow \cT_2^c$. The restriction of $\cC$ to the preimage of $\cT_2$ in $\C$ is the universal family over $\cT_2$ (see \cite{mess1992torelli}); we shall denote it by $\C'\rightarrow \cT_2$. It is a $C^{\infty}$ fiber bundle over $\cT_2$ with standard fiber $S_2$.
\begin{proof}[Proof of Proposition \ref{proposition8}]
The moduli map $\h_2\rightarrow \cT_2^c$ is the holomorphic map which sends a period matrix $\Omega_0$ to the isomorphism class $[C;F]$ of framed curves with period matrix $\Omega_0$. This is simply the inverse map of the period map $\pi:\cT_2^c\rightarrow \h_2$, hence it is a biholomorphism. Since the family $f:\Theta\rightarrow \h_2$ is recovered from $\C\rightarrow \cT_2^c$ via pullback along $\pi^{-1}$, there is an isomorphism of families of framed curves $\Theta\xrightarrow{\cong} \C$.

\end{proof}
The fact that $\C$ is isomorphic to $\Theta$ will be of considerable utility in the sequel. Using facts about theta functions, we will be able to explicitly describe the geometry of $\Theta$ and its universal cover, and therefore that of $\C$ and $\chat$.  
\\\\\indent 
One of the first such results concerns the geometry of the locus $\C^{red}$ of reducible fibers of $\C\rightarrow \cT_2^c$. Define $\Theta^{red}$ to be the preimage of $\h_2^{red}$ in $\Theta$ under the projection $\Theta\rightarrow \h_2$. Recall that a divisor with \emph{simple normal crossings} is a normal crossings divisor all of whose irreducible components are smooth. Define 
$\Theta^{red}_{\beta}$ to be the component of $\Theta^{red}$ lying over the component $\h_{2,\beta}^{red}$ of $\h_2^{red}$. 
\begin{proposition}
The locus $\C^{red}$ of singular fibers in $\C$ is a divisor with simple normal crossings.
\end{proposition}
\begin{proof}
The isomorphism $\Theta\cong \C$ induces a biholomorphism $\Theta^{red}\cong \C^{red}$. We will prove that $\Theta^{red}$ has simple normal crossings. Our strategy is to prove the result for the locus $\Theta^{red}_{\h_1\times\h_1}$ lying over the connected component $\h_1\times \h_1\subset \h_2^{red}\subset \h_2$ and then use the action of $\Sp_2(\ZZ)$ on $\mathfrak{X}_2$ to transplant the result to other components of $\h_2^{red}$.

The locus $\Theta^{red}_{\h_1\times \h_1}$ is clearly an analytic subset of $\Theta$ of codimension 1; it is defined by the single  equation $f_{12} = 0$, where $f_{12}$ denotes the $\Omega_{12}$ component of the projection $f: \Theta\rightarrow \h_2$, viewed as a 3-tuple of holomorphic functions. We may identify $\Theta^{red}_{\h_1\times \h_1}$ with the zero locus of $\vartheta$ inside of $\mathfrak{X}_2|_{\h_1\times \h_1}$. When $\Omega_{12} = 0$, there is a product decomposition
\begin{equation}
\vartheta(\Omega,z) = \vartheta_{(1/2,1/2)}(\Omega_{11},z_1)\vartheta_{(1/2,1/2)}(\Omega_{22},z_2).
\end{equation}
Now $\Theta^{red}_{\h_1\times \h_1}$ is the union of the zero loci of $g_j(\Omega,z):=\vartheta_{(1/2,1/2)}(\Omega_{jj},z_j)$ inside of $\mathfrak{X}_2|_{\h_1\times \h_1}$. Both components are smooth because, in local coordinates $(\Omega,z)$ the gradients of the functions $\vartheta_{(1/2,1/2)}(\Omega_{jj},z_j)$ are non-vanishing along their zero loci. This is a consequence of the fact that the Jacobi theta function $\vartheta_{(1/2,1/2)}(\tau,x)$ has only simple zeros. Thus $\Theta^{red}_{\h_1\times \h_1}$ is the union of two smooth components. 

We now show that the locus where these components intersect is locally the union of the hyperplanes $z_1 = 0$ and $z_2 = 0$ inside of $\h_2\times \CC^2$. 
The two components meet along the locus $z_1=z_2 = 0$ in $\mathfrak{X}_2|_{\h_1\times \h_1}$. Since this is local problem, we will work in a sufficiently small neighborhood of $x = (\tau_1\oplus \tau_2, 0)\in \h_2\times \CC^2$, as this is a point in the preimage of the intersection of the two components. Locally around $x$ we have an identity
\begin{equation}
 \vartheta_{(1/2,1/2)}(\Omega_{11},z_1)\vartheta_{(1/2,1/2)}(\Omega_{22},z_2)  = z_1z_2\cdot u
\end{equation}
where $u$ is a non-vanishing holomorphic function. The claim now follows. 

Choose an element $M = \left(\begin{array}{cc}A & B \\C & D\end{array}\right)\in \Sp_2(\ZZ)$ such that $M$ carries the component $\h_1\times \h_1$ to $\h_{2,\beta}^{red}$. Then the automorphism of $\mathfrak{X}_2$ defined by 
\begin{equation}
[\Omega,z]\rightarrow [M\cdot \Omega, z\cdot(C\Omega+D)^{-1}]
\end{equation}
carries $\mathfrak{X}_2|_{\h_1\times \h_1}$ to $\mathfrak{X}_2|_{\h_{2,\beta}^{red}}$. The theta function $\vartheta_{\alpha}(\Omega,z)$ pulls back to 
\begin{equation}
\vartheta_{\alpha}(M\Omega,z\cdot (C\Omega+D)^{-1}) = \hat u\cdot \vartheta_{\alpha'}(\Omega,z)
\end{equation}
where $\hat u$ is a specific non-vanishing holomorphic function, and $\alpha'$ is a specific characteristic derived from $\alpha$ (see, for example, \cite[p.227]{birkenhake2004complex}). As we have argued, the zero locus of this function in $\mathfrak{X}_2|_{\h_1\times h_1}$ is a divisor with simple normal crossings. It follows that $\Theta^{red}_{\beta}$ is a divisor with simple normal crossings in $\Theta$. Since $\Theta^{red}$ is the union of all of the $\Theta^{red}_{\beta}$, all of which are disjoint, $\Theta^{red}$ is a divisor with simple normal crossings.
\end{proof}
Let $X_{\beta}$ denote a component of $\Theta^{red}_{\beta}$. The projection $\Theta\rightarrow \h_2$ restricts to a projection $X_{\beta}\rightarrow \h_{2,\beta}^{red}$. 
\begin{lemma}
The projection $X_{\beta}\rightarrow \h_{2,\beta}^{red}$ is a family of elliptic curves.  It is a topologically trivial bundle with fiber a 2-torus. 
 \end{lemma}
\begin{proof}

The idea is to show that Ehremann's Theorem \cite{lamotke1981topology} is applicable. First observe that $\Theta^{red}_{\beta}\rightarrow \h_{2,\beta}^{red}$ is proper. Since $X_{\beta}$ is a closed subset of $\Theta^{red}_{\beta}$, it follows that $X_{\beta}\rightarrow \h_{2,\beta}^{red}$ is proper. We now argue that $X_{\beta}\rightarrow \h_{2,\beta}^{red}$ is a submersion. For simplicity, we work with a component $X_{\h_1\times \h_1}$ over $\h_1\times \h_1$ of $\h_2^{red}$. The argument can be transferred to the other components of $\h_2^{red}$ by symmetry. Since the problem is local, we work with the component $Z$ of the zero locus of the holomorphic function $f(\Omega_{11}\oplus \Omega_{22},z) = \vartheta_{(1/2.1/2)}(\Omega_{11},z_1)$ in $\h_1\times \h_1\times \CC^2$ all of whose points are of the form $(\Omega_{11}\oplus \Omega_{22}, 0 ,x)$ where $x\in \CC$. 

Fix a point $P = (\tau_1\oplus\tau_2, 0 ,x)\in Z$. Define a map $\h_1\times \h_1\rightarrow \h_1\times \h_1\times \CC^2$ by 
$$\Omega_{11}\oplus\Omega_{22}\rightarrow (\Omega_{11}\oplus\Omega_{22}, 0, x).$$
This gives a section of the projection $Z\rightarrow \h_1\times \h_1$ passing through the point $P$. This implies that $Z\rightarrow \h_1\rightarrow \h_1$ is a submersion, and hence that the projection $X_{\beta}\rightarrow \h_{2,\beta}^{red}$ is a submersion. 

Observe that the fiber in $X_{\beta}$ over a point $\tau_1\oplus\tau_2\in \h_1\times \h_1$ is the zero locus of the function $\vartheta_{(1/2,1/2)}(\tau_1,z_1)$ in the complex torus $\CC^2/\Lambda(\tau_1\oplus\tau_2)$. This is just the subtorus $\CC/\Lambda(\tau_2)$ of $\CC^2/\Lambda(\tau_1\oplus \tau_2)$ consisting of points of the form $(0,x)$. 

An application of Ehresmann's Theorem shows that $X_{\h_1\times \h_1}\rightarrow \h_1\times \h_1$ is a trivial fiber bundle with standard fiber a 2-torus.

\end{proof}

\begin{remark}
\emph{
It should also be noted that the singular locus $\Theta^{red,sing}$ of $\Theta^{red}$ is precisely the locus along which the components of $\Theta^{red}$ intersect. Each irreducible component of $\Theta^{red,sing}$ is isomorphic to $\h_1\times\h_1$. 
The projection $\Theta\rightarrow \h_2$ induces a submersion of the unique component of $\Theta^{red,sing}_{\beta}$ lying over $\h_{2,\beta}^{red}$ onto $\h_{2,\beta}^{red}$.
}
\end{remark}
In a later section, we will invoke certain results from stratification theory (see \cite{goresky1988stratified}). As a consequence of our work so far we state the following proposition.
\begin{proposition}
The filtration $\Theta\supset \Theta^{red}\supset \Theta^{red,sing}$ induces a Whitney stratification of $\Theta$. With respect to this stratification and the Whitney stratification of $\h_2$ induced by $\h_2\supset \h_2^{red}$, the projection $\Theta\rightarrow \h_2$ is a stratified submersion.
\end{proposition}
\subsection{Geometry of the Universal Cover of $\Theta$}

By Mess's work \cite{mess1992torelli}, $\cT_2$ is homotopy equivalent to an infinite bouquet of circles. The generators of $\pi_1(\cT_2,p)$ have the following description. Identify $\cT_2$ with $\h_2-\h_2^{red}$ via the period map. Let $\Delta_{\beta}$ be a small holomorphic disk in $\h_2$ that is transverse to $\h_{2,\beta}^{red}$ at $p_{\beta}$. 
Then for each $\beta$ there exists a continuous path $\gamma_{\beta}$ from $\Omega_0$ to a point on the boundary of $\Delta_{\beta}$ such that the oriented loops
\begin{equation}
\gamma_{\beta}^{-1}\cdot \partial \Delta_{\beta}\cdot \gamma_{\beta}
\end{equation}
freely generate $\pi_1(\h_2-\h_2^{red},\Omega_0)$. The Torelli group $T_2$ is freely generated by the monodromy of $\C-\C^{red}$ around these loops.
\\\\\indent 
Let us now fix some notation. These conventions will be in effect for the rest of the paper. Let $\pi = \pi_1(S_2,*)$ denote the fundamental group of the reference surface $S_2$ of genus 2. Let $\pi'\vartriangleleft \pi$ denote the commutator subgroup, and define $H: =\pi/\pi'\cong H_1(S_2,\ZZ)$ denote the abelianization of $\pi$.

\begin{lemma}
The embedding $\cC\hookrightarrow \mathfrak{X}_2$ induces an isomorphism on fundamental groups $\pi_1(\cC,*)\rightarrow \pi_1(\mathfrak{X}_2,*)$.
\end{lemma}
\begin{proof}
Consider the universal family $\C'\rightarrow \cT_2$ (see \cite{arbarello2011geometry}, p. 461).  This is a $C^{\infty}$ fiber bundle $S_2\rightarrow \C'\rightarrow \cT_2$. The fundamental group of $\C'$ is isomorphic to the pointed Torelli group $T_{2,1}$ (see \cite{farb2011primer}). The only non-trivial segment of the associated long exact sequence of homotopy groups is
\begin{equation}\label{exactsequencetorelli}
1\rightarrow \pi\rightarrow T_{2,1}\rightarrow T_2\rightarrow 1
\end{equation}
%
%
%
%
%
Since $T_2$ is free by \cite{mess1992torelli}, the exact sequence (\ref{exactsequencetorelli}) splits. This allows us to view $T_2$ as a subgroup of $T_{2,1}$ and $T_{2,1}$ as a semi-direct product. 

Observe that the inclusion $\C'\hookrightarrow \mathfrak{X}_2$ factors as $\C'\hookrightarrow \C\hookrightarrow \mathfrak{X}_2$. Notice also that this inclusion induces a surjection $\pi_1(\C',*)\rightarrow \pi_1(\mathfrak{X}_2,*)$. This latter statement follows from the fact that the fibers of $\C'$ are embedded in the fibers of $\mathfrak{X}_2$ as theta divisors. 

Consider the homomorphism $i:\pi_1(\C',*)\rightarrow \pi_1(\C,*)$ induced by inclusion. This is a surjection, since $\C'$ is the complement in $\C$ of the normal crossings divisor $\C^{red}$. Elements of $T_2$, viewed as a subgroup of $T_{2,1}$, lie in the kernel of $i$ because $\h_2$ is contractible. Elements in the commutator subgroup $\pi'\vartriangleleft \pi\vartriangleleft T_{2,1}$ also lie in the kernel of $i$ because, by Mess's description of the monodromy, any separating simple closed curve on any fiber of $\C'$ is freely homotopic to a vanishing cycle of $\C$. Since the vanishing cycles are themselves nullhomotopic, it follows that any separating simple closed curve on a fiber of $\C'$ is as well. Since any element of $\pi'$ can be written as a product of separating simple closed curves, it follows that $\pi'$ lies in the kernel. Thus the subgroup $\pi' \rtimes T_2$ of  $T_{2,1}$ lies in the kernel of $i$.

There is a surjective homomorphism $\pi\rtimes T_2\rightarrow \pi^{ab}\cong H$ defined by $ct\rightarrow \overline{c}$, where $t\in T_2$ and $c\in \pi$. The kernel of this homomorphism is precisely $\pi'\rtimes T_2$. This induces an isomorphism $(\pi\rtimes T_2)/(\pi'\rtimes T_2)\cong H$. Thus we have obtained a sequence of group homomorphisms
\begin{equation}
H\rightarrow \pi_1(\C,*)\rightarrow \pi_1(\mathfrak{X}_2,*)
\end{equation}
such that the composition $H\rightarrow \pi_1(\mathfrak{X}_2,*)\cong H$ is surjective. Since $H$ is a finitely generated abelian group, this composition is actually an isomorphism. Furthermore, since $H\rightarrow \pi_1(\C,*)$ is surjective, the homomorphism $\pi_1(\C,*)\rightarrow H$ is injective, hence it is an isomorphism. 
\end{proof}
We shall call upon the following corollary in a later section.
\begin{corollary}
The inclusion $\C_{\Omega_0}\hookrightarrow \C$ induces an isomorphism $H_1(\C_{\Omega_0})\xrightarrow{\cong} H_1(\C)$, where $\C_{\Omega_0}$ is any fiber.
\end{corollary}

We are now in a position to describe the geometry of the universal cover $\chat$ of $\C$.  We will do this by examining the geometry of the universal cover $\Thetatilde$ of $\Theta$. As a corollary of the preceding proposition, we see that the inclusion $\Theta\hookrightarrow \mathfrak{X}_2$ induces an isomorphism of fundamental groups. From this it follows that the preimage of $\Theta$ in $\h_2\times \CC^2$, i.e. the zero locus of $\vartheta_{\alpha}$ in $\h_2\times \CC^2$, is a connected covering space of $\Theta$ with covering group $H$. In other words, it is the universal cover of $\Theta$. We will denote this space by $\widetilde\Theta$. It follows that there is an isomorphism $\chat\rightarrow \Thetatilde$ and that it respects the fibers of the projections $\chat\rightarrow \C\rightarrow \h_2$ and $\Thetatilde\rightarrow \Theta\rightarrow \h_2$. We record this as 
\begin{lemma}
There is a biholomorphism $\chat\rightarrow \Thetatilde$ which preserves the fibers of the projections $\chat\rightarrow \h_2$ and $\Thetatilde\rightarrow \h_2$.
\end{lemma}

Let us now describe the fibers of $\chat\rightarrow \h_2$ in some detail.  Since the inclusion of a fiber $\C_{\Omega}\hookrightarrow \chat$ induces a surjection on the level of fundamental groups, it follows that the preimage of $\C_{\Omega}$ in $\chat$ is connected. Since this preimage is precisely the fiber $\chat_{\Omega}$ of $\chat$ over $\Omega$. It follows that the fibers of $\chat_{\Omega}$ are connected. Since the covering $\chat\rightarrow \C$ is an $H$-covering, it follows that the induced map $\chat_{\Omega}\rightarrow \C_{\Omega}$ is an $H$-covering. Since the fundamental group $\pi_1(\C_{\Omega},*)$ is a finitely generated group with abelianization $H$, it follows that $\chat_{\Omega}$ is the universal abelian covering of $\C_{\Omega}$. We record this in the following
\begin{lemma}
The fiber $\chat_{\Omega}$ of $\chat$ over $\Omega\in \h_2$ is the universal abelian cover of the fiber $\C_{\Omega}$ of $\C$ over $\Omega$. Concretely, the fiber $\chat_{\Omega}$ is isomorphic to the zero locus of the theta function $\vartheta_{\alpha}(\Omega,-):\CC^2\rightarrow \CC$. 
\end{lemma}
Now define $\Thetatilde^{red}$ to be the preimage of $\Theta^{red}$ under the universal covering $\Thetatilde\rightarrow \Theta$. Let $\Thetatilde^{red}_{\beta}$ denote the preimage of $\Theta^{red}_{\beta}$ in $\Thetatilde$. Then $\Thetatilde^{red}_{\beta}$ is a divisor in  $\Thetatilde$ with simple normal crossings. It is easiest to see this by considering the locus $\Thetatilde^{red}_{\h_1\times \h_1}$ over $\h_1\times \h_1$. This is isomorphic to the zero locus of $\vartheta_{\alpha}(\Omega_{11},z_1)\vartheta_{\alpha}(\Omega_{22},z_2)$ inside of $\h_1\times \h_1\times \CC^2$. This is the subvariety of $\h_1\times \h_1\times\CC^2$ given by the union of the components $D_1(m,n)$ and $D_2(m,n)$, where $m,n\in \ZZ$, with equations
\begin{equation}
D_j(m,n) = \{(\Omega_{11},\Omega_{22}, z_1,z_2): z_j = m+n\Omega_{jj}\}
\end{equation}
It is a divisor with simple normal crossings, and each irreducible component is isomorphic to $\h_1\times \h_1\times \CC$. Note that, for fixed $j$, the components $D_j(m,n)$ are acted upon freely and transitively by $\pi_1(\mathfrak{X}_2,*)/\pi_1(X_j,*)\cong \ZZ^2$, where $X_j$ is the component of $\Theta^{red}_{\h_1\times\h_1}$ whose preimage is the union of the $D_j(m,n)$.

By symmetry, the other loci $\Thetatilde^{red}_{\beta}$ are seen to have an essentially identical description. 

Also note that, much in the same way that the projection $\Theta\rightarrow \h_2$ restricts to a submersion $X_{\beta}\rightarrow \h_{2,\beta}^{red}$, the projection $\Thetatilde\rightarrow \h_2$ maps each irreducible component of $\Thetatilde^{red}_{\beta}$ submersively onto $\h_{2,\beta}^{red}$. Additionally, each component of the singular locus $\Thetatilde^{red,sing}_{\beta}$ of $\Thetatilde^{red}_{\beta}$ is isomorphic to $\h_1\times \h_1$ and mapped submersively onto $\h_{2,\beta}^{red}$ by the projection $\Thetatilde\rightarrow \h_2$. Note that set of these components are acted on freely and transitively by the deck group $H$.

\begin{lemma}
The locus $\Thetatilde^{red}$ of reducible fibers in $\Thetatilde$ is a divisor with simple normal crossings. Its irreducible components are isomorphic to $\h_1\times \h_1\times \CC$. The singular locus $\Thetatilde^{red,sing}$ is smooth, and its irreducible components are isomorphic to $\h_1\times \h_1$.
\end{lemma}

Let $S_2^{ab}$ denote the universal abelian cover of $S_2$. As a corollary of the fact that $\Theta-\Theta^{red}\cong \C'$ is a fiber bundle over $\h_2-\h_2^{red}$ with fiber $S_2$, we have that $\Thetatilde-\Thetatilde^{red}$ is a fiber bundle over $\h_2-\h_2^{red}$ with fiber $S_2^{ab}$. If $\chat^{red}$ denotes the preimage of $\C^{red}$ in $\chat$, then the same statement applies to $\chat-\chat^{red}$ as well.

The preceding discussion is summarized in the following proposition.
\begin{proposition}
The filtration $\Thetatilde\supset \Thetatilde^{red}\supset \Thetatilde^{red,sing}$ induces a Whitney stratification of $\Thetatilde$. With respect to this stratification and the Whitney stratification of $\h_2$ induced by $\h_2\supset \h_2^{red}$, the projection $\Thetatilde\rightarrow \h_2$ is a stratified submersion.
\end{proposition}
\section{Computing $H_k(\widetilde{\C})$ for $k \geq 4$}
In this section, we will show that the integral homology groups $H_k(\chat)$ vanish when $k\geq 4$. This will be done by considering the PL Gysin sequence derived in Hain's paper \cite{hain2006finiteness}. In our case, this is an exact sequence which relates the homology groups of $\chat$ and $\chat-\chat^{red}$ to the compactly supported cohomology of $\chat^{red}$. 
Also in \cite{hain2006finiteness}, a spectral sequence was derived which will allow us to compute $H_c^k(\chat^{red})$, given certain combinatorial input that describes the way in which the components of $\chat^{red}$ intersect.
\\\\\indent
We begin with the following statement.
\begin{lemma}\label{lemma26}
The total space of the fiber bundle 
$$S_2^{ab}\rightarrow \left(\chat-\chat^{red}\right)\rightarrow \cT_2$$ is a $K(G,1)$ space. The homology groups
$H_k(\chat-\chat^{red})$ vanish for $k\geq 3$.
\end{lemma}
\begin{proof}
Since the fiber and base of this fiber bundle are $K(G,1)$ spaces, the total space is as well. The group $\Gamma:=\pi_1(\chat-\chat^{red},*)$ is an extension
\begin{equation}
1\rightarrow \pi'\rightarrow \Gamma\rightarrow T_2\rightarrow 1.
\end{equation}
where $\pi'$ denotes the commutator subgroup of $\pi = \pi_1(S_2,*)$. Recall that $\pi'$ is a free group. Since $\pi'$ and $T_2$ are free, and therefore have cohomological dimension 1, the group $\Gamma$ has cohomological dimension at most 2 (see \cite{brown1982cohomology}, p. 188). By the universal coefficients theorem, this implies that $H_k(\Gamma,\mathbb{Z}) = 0$ for $k\geq 3$. Finally, we obtain isomorphisms 
$H_k(\chat-\chat^{red})\cong H_k(\Gamma,\mathbb{Z}) = 0$ for $k\geq 3$.
\end{proof}
\subsection{The PL Gysin Sequence}
We now describe the long exact sequence constructed in \cite{hain2006finiteness}. Let $X$ be an oriented PL manifold of dimension $m$, and let $Y\subset X$ be a closed PL subset (see \cite{borel2008intersection}, p.2). Then there is a long exact sequence
\begin{equation}
\cdots\rightarrow H_c^{m-k-1}(Y)\rightarrow H_k(X-Y)\rightarrow H_k(X)\rightarrow H_c^{m-k}(Y)\rightarrow \cdots
\end{equation}
The groups $ H_c^{\bullet}(Y)$ can be computed, in principle, using the spectral sequence that was constructed in the same paper. We describe it now.

Suppose that $Y$ is a locally finite union 
\begin{equation}
Y = \bigcup_{i\in I}Y_i
\end{equation}
of closed PL subspaces of the manifold $X$, where $I$ is a partially ordered set. Define
\begin{equation}
Y_{(i_0,\ldots, i_k)} = Y_{i_0}\cap\cdots \cap Y_{i_k}.
\end{equation}
and set 
\begin{equation}
\cY_k = \coprod_{i_0<\cdots <i_k}Y_{(i_0,\ldots, i_k)}.
\end{equation}

Then there is a spectral sequence 
\begin{equation}
E_1^{s,t} = H_c^t(\cY_s)\implies H_c^{s+t}(Y)
\end{equation}
where $\cY_s = \displaystyle \coprod_{\alpha_0<\cdots <\alpha_s} Y_{\alpha_0}\cap\cdots \cap Y_{\alpha_s}$.

We will apply these tools in the case where $X = \chat$ and $Y = \chat^{red}$.
At this point, we need to describe the combinatorics of the ways in which the components of $\Thetatilde^{red}$ intersect. Note that the components of the form $D_1(m,n)$ do not intersect each other; they only intersect components of the form $D_2(m,n)$. Similarly, the components of the form $D_2(m,n)$ only intersect the components of the form $D_1(m,n)$. 
Just as the components of $\Thetatilde^{red}_{\h_1\times \h_1}$ fall into two types, so do the components of $\Thetatilde^{red}_{\beta}$ for an arbitrary component $\h_{2,\beta}^{red}$ of $\h_2^{red}$. Two components of $\Thetatilde^{red}_{\beta}$ are of the same type if and only if the deck group $H$ carries one to the other. Arbitrarily label the components of $\Thetatilde^{red}_{\beta}$ of one type with a ``1''; label the components of $\Thetatilde^{red}_{\beta}$ of the other type with a ``2". Now $\Thetatilde^{red}_{\beta}$ is the union of the components $\widetilde D_{\beta,1}(m,n)$ and $\widetilde D_{\beta,2}(m,n)$.

\begin{lemma}\label{lemma27}
The groups $H_c^k(\chat^{red})$ vanish unless $k = 5$ or $k=6$. 
\end{lemma}
\begin{proof}
Place any partial order on the set of irreducible components of $\chat^{red}$. Observe that two components $\widetilde D_{\beta_1,j}(m_1,n_1)$ and $\widetilde D_{\beta_2,k}(m_2,n_2)$ intersect if and only if $\beta_1 = \beta_2$ and $j\neq k$. This implies that no more than two components of $\chat^{red}$ can intersect non-trivially. Otherwise two components of $\chat^{red}_{\beta}$ of the same type would intersect; but this is impossible. Therefore the subsets $\cY_s$ are all empty for $s\geq 2$. 

As has already been established, the components of $\cY_0$ are analytically isomorphic to $\h_1\times \h_1\times \CC$, hence diffeomorphic to $\RR^6$. The components of $\cY_1$ are analytically isomorphic to $\h_1\times\h_1$, hence diffeomorphic to $\RR^4$. Thus the $E_1$ page of the spectral sequence for $H_c^{\bullet}(\chat^{red})$ is of the following form:
\begin{equation*}
\begin{tabular}{l|r}
 & \xymatrix{
  H^6_c(\cY_0)&0  \\
 0 &0 \\
  0&H_c^4(\cY_1) \\
  \vdots& \vdots    \\
  0 & 0
  }\\
\hline\\
\end{tabular}
\end{equation*}
Because of the  shape of this spectral sequence, all differentials vanish and this spectral sequence degenerates at the $E_1$ page. It follows that $H_c^k(\chat^{red}) = 0$ unless $k = 5$ or $k = 6$.
\end{proof}
We now combine this result with the vanishing of the homology of $\chat-\chat^{red}$ in degrees at least 3 to obtain, by use of the PL Gysin sequence, the following result.
\begin{proposition}\label{prop22}
The group $H_k(\chat)$ vanishes for $k\geq 4$.
\end{proposition}
\begin{proof}
For each $k$, we have exact sequences
\begin{equation}
H_c^{7-k}(\chat^{red})\rightarrow H_k(\chat-\chat^{red})\rightarrow H_k(\chat)\rightarrow H_c^{8-k}(\chat^{red})
\end{equation}
coming from the PL Gysin sequence. By Lemma \ref{lemma27}, when $k\geq 4$, the groups $H_c^{7-k}(\chat^{red})$ and $H_c^{8-k}(\chat^{red})$ are both zero, since then $7-k$ and $8-k$ are both $\leq 4$. Furthermore, $H_k(\chat-\chat^{red}) = 0$ for $k\geq 4$ by Lemma \ref{lemma26}. This implies that $H_k(\chat) = 0$ for $k \geq 4$ by exactness.
\end{proof}
\section{\texorpdfstring{A Direct Sum Decomposition of $H_2(\chat-\chat^{red})$}{A Direct Sum Decomposition of H2(chat-Dhat)}}

It will be convenient for our purposes to replace $\chat-\chat^{red}$ with a homotopy equivalent space. Specifically, we would like to view it, up to homotopy equivalence, as a $S_2^{ab}$ bundle over a 1-complex. We construct the 1-complex that will serve as the base as follows.

View the real line as a CW complex having 0-cells at the even integers, and 1-cells joining the integer $2n-2$ to $2n$ for each $n\in \ZZ$. Now attach a circle $S^1_{2n}$ at each even integer $2n$. The resulting space is homotopy equivalent to an infinite bouquet of circles. Choose some indexing $\beta(n)$ of the components $\h_{2,\beta}^{red}$ by $\ZZ$. 

Now recall the description of Mess's results given in Section 1.1.2 above. Define a continuous map $f$ from $X$ to $\h_2-\h_2^{red}$ in the following way. On the interval $[2n-1,2n]$ define $f$ by $f(t) = \gamma_{\beta(n)}(t)$; on the circle $S^1_{2n}$, define $f$ to be an orientation-preserving homeomorphism $S^1_{2n}\rightarrow \partial \Delta_{\beta}$ which sends the point $2n\in S^1_{2n}$ to the point where $\gamma_{\beta}$ meets $\partial \Delta_{\beta}$; on the interval $[2n,2n+1]$, define $f$ by $f(t) = \gamma_{\beta(n)}(1-t)$. The map $f:X\rightarrow \h_2-\h_2^{red}$ is a homotopy equivalence.

The bundle $\chat-\chat^{red}\rightarrow \cT_2$ can now be pulled back to $X$. Call the pullback bundle $\chatx$. The total space of this bundle is homotopy equivalent to $\chat-\chat^{red}$. Let $E = \displaystyle \coprod_{n\in \ZZ} S^1_{2n}\subset X$. Denote the restriction of $\chatx$ to $E$ by $\chate$. Away from the point $2n\in S^1_{2n}$, choose a fiber $F_n$ of $\chatx|_{S^1_{2n}}$ and define $F = \displaystyle \coprod_{n\in \ZZ} F_n\subset \chatx$.

\begin{proposition}
There is a direct sum decomposition
\begin{equation}
H_2(\chat-\chat^{red}) = \cF\oplus \bigoplus_{\beta}H_2(\chat_{\partial \Delta_{\beta}})
\end{equation}
where $\cF$ is a free abelian group.
\end{proposition}
\begin{proof}

The inclusion $\chate\hookrightarrow \chat-\chat^{red}$ factors through the inclusion $\displaystyle \coprod_{\beta} \chat_{\partial \Delta_{\beta}}\subset \chat-\chat^{red}$, and induces a homotopy equivalence $\chate\simeq \displaystyle\coprod_{\beta} \chat_{\partial \Delta_{\beta}}$.
By excision, there are natural isomorphisms $H_k(\chatx,\chate)\cong H_k(\chatx-F,\chate-F)$ for each $k$.

The long exact sequence of homology for the pair $(\chatx,\chate)$ has a segment
\begin{equation}
H_3(\chatx,\chate)\rightarrow  H_2(\chate)\rightarrow H_2(\chatx)\rightarrow H_2(\chatx,\chate)
\end{equation}
which we study in several steps. 
\begin{enumerate}
\item[\underline{Step 1}]: First we show that $H_3(\chatx,\chate) = 0$. We will study the exact sequence of the pair $(\chatx-F,\chate-F)$. Since $\chatx-F$ is a fiber bundle over a contractible CW-complex, it is a trivial bundle. It therefore has the homotopy type of a single fiber, which we denote by $F_0$. Since $F_0\cong S_2^{ab}$ is homotopy equivalent to a 1-complex, it follows that $\chatx-F$ is homotopy equivalent to a 1-complex. Similarly, each connected component of $\chate-F$ is a $S_2^{ab}$-bundle over a contractible CW-complex. It follows that $\chate-F$ is homotopy equivalent to a 1-complex.

\item[\underline{Step 2}]:
The groups $H_3(\chatx-F)$ and $H_2(\chate-F)$ vanish by Step 1. It follows that the group $H_3(\chatx-F,\chate-F)$ also vanishes. On the other hand, the group $H_2(\chatx-F)$ vanishes and the group $H_1(\chate-F)$ is free abelian. This implies that $H_2(\chatx-F,\chate-F)$ is a subgroup of  $H_1(\chate-F)$ and is therefore also free abelian. Finally, this implies that $H_2(\chatx,\chate)\cong H_2(\chatx-F,\chate-F)$ is free abelian. 

\item[\underline{Step 3}]:
The conclusions arrived at in Step 2 imply that the sequence
\begin{equation*}
0\rightarrow H_2(\chate)\rightarrow H_2(\chatx)\rightarrow H_2(\chatx,\chate)
\end{equation*}
is exact and that the image of the rightmost map is a free abelian group, which we denote by $\cF'$. This implies that there is a split exact sequence
\begin{equation*}
0\rightarrow H_2(\chate)\rightarrow H_2(\chatx)\rightarrow \cF'\rightarrow 0.
\end{equation*}

By functoriality of the long exact sequence of homology for a pair, there is a commutative diagram 
\begin{equation*}
\xymatrix{
0\ar[r]& H_2(\chate)\ar[r]\ar[d]^{\cong}& H_2(\chatx) \ar[d]^{\cong}\ar[r] &\cF'\ar[r] &0\\
0\ar[r]& H_2(\coprod_{\beta} \chat_{\partial \Delta_{\beta}})\ar[r] & H_2(\chat-\chat^{red})\ar[ru] & 
}
\end{equation*}
with exact rows. This implies that the corresponding sequence
\begin{equation*}
0\rightarrow H_2(\coprod_{\beta} \chat_{\partial \Delta_{\beta}})\rightarrow H_2(\chat-\chat^{red})\rightarrow \cF' \rightarrow 0 
\end{equation*}
is exact. The result follows.
\end{enumerate}
\end{proof}

\section{The Computation of \texorpdfstring{$H_3(\chat)$ and $H_2(\chat)$}{The computation of H3chat and H2chat}}
\ 
\indent Our tool of choice in this section is the long exact sequence of homology for the pair $(\chat,\chat-\chat^{red})$. Recall that we have already shown that $H_k(\chat) = 0$ when $k\geq 4$ and that $H_k(\chat-\chat^{red}) = 0$ when $k\geq 3$. We are thus left to consider the sequence
\begin{equation*}
0\rightarrow H_3(\chat)\rightarrow H_3(\chat,\chat-\chat^{red})\rightarrow H_2(\chat-\chat^{red})
\rightarrow H_2(\chat)\rightarrow H_2(\chat,\chat-\chat^{red})
\end{equation*}
We will show that the boundary map $\partial: H_3(\chat,\chat-\chat^{red})\rightarrow H_2(\chat-\chat^{red})$ factors in the following way:
\begin{equation}
\xymatrix{
H_3(\chat,\chat-\chat^{red})\ar[rr]^{\displaystyle\partial}\ar[rd]_{\cong}& &H_2(\chat-\chat^{red})\\
&\bigoplus_{\beta}H_2(\chat_{\beta})\ar@{^{(}->}[ru] &
}
\end{equation}

This will imply that $H_3(\chat)$ vanishes and, coupled with the fact that $H_2(\chat,\chat-\chat^{red})$ is a free abelian group (generated by small disks transverse to the irreducible components of $\chat^{red}$), that $H_2(\chat)$ is a free abelian group.
\\\\\indent 
In order to carry this out, we will first find a convenient direct sum decomposition for $H_3(\chat, \chat-\chat^{red})$ which is in some sense compatible with the boundary map. Choose a closed tubular neighborhood $\overline{N}$ of $\h_2^{red}$. This restricts to a tubular neighborhood $\overline{N}_{\beta}$ of each component $\h_{2,\beta}^{red}$ of $\h_2^{red}$. Let $N = \overline{N}-\partial{\overline{N}}$. By excision, there is a natural isomorphism
\begin{equation}
H_3(\chat, \chat-\chat^{red})\cong \bigoplus_{\beta}H_3(\chat_{N_{\beta}},\chat_{N_{\beta}} - \chat^{red}_{\beta}).
\end{equation}
To prove the desired result, we will need to gain a thorough understanding of the terms $H_3(\chat_{N_{\beta}},\chat_{N_{\beta}}-\chat^{red}_{\beta})$. 

Fix a component $\h_{2,\beta}^{red}$ of $\h_2^{red}$. We will study the topology of the pair of spaces $(\chat_{N_{\beta}},\chat_{N_{\beta}} - \chat^{red}_{\beta})$. Our goal is to show that there is a homeomorphism
\begin{equation}
(\chat_{N_{\beta}},\chat_{N_{\beta}} - \chat^{red}_{\beta})\cong \h_{2,\beta}^{red} \times (\chat_{N(p_{\beta})},\chat_{N^*(p_{\beta})}).
\end{equation}
Here $N(p_{\beta})$ is the fiber of $N_{\beta}$ over the point $p_{\beta}\in \h_{2,\beta}^{red}$, $\chat_{N(p_{\beta})}$ is its preimage in $\chat$ and $\chat_{N^*(p_{\beta})}$ is the preimage of $N^*(p_{\beta}):= N(p_{\beta})-p_{\beta}$ in $\chat$. We may assume that $N(p_{\beta}) = \Delta_{\beta}$.
\\\\\indent 
In order to prove this result, it appears easier to first work with the pair of spaces $(\Theta_{\overline{N}_{\beta}}, \Theta_{\overline{N}_{\beta}}-\Theta^{red}_{\overline{N}_{\beta}})$.
The immediate goal is to show that the sequence of projections $\Theta_{\overline{N}_{\beta}}\rightarrow \overline{N}_{\beta}\rightarrow \h_{2,\beta}^{red}$ fibers the pair $(\Theta_{\overline{N}_{\beta}},\Theta_{\overline{N}_{\beta}} - \Theta^{red}_{\beta})$ locally trivially. We will use Thom's First Isotopy Lemma (see \cite[p.41]{goresky1988stratified}) to accomplish this. For the reader's convenience, we reproduce it here.
\begin{theorem}[Thom's First Isotopy Lemma]
Let $f: Z\rightarrow \RR^n$ be a proper stratified submersion. Then there is a stratum-preserving homeomorphism
\begin{equation}
h: Z\rightarrow \RR^n\times \left(f^{-1}(0)\cap Z\right)
\end{equation}
which is smooth on each stratum and commutes with the projection to $\RR^n$. In particular the fibers of $f|_Z$ are homeomorphic by a stratum-preserving homeomorphism.
\end{theorem}
By radially shrinking $\overline{N}$ if needed, we may assume that $N_{\beta}$ is contained in a larger open tubular neighborhood $N'_{\beta}$ of $\h_{2,\beta}^{red}$ such that the projection $N_{\beta}'\rightarrow \h_{2,\beta}^{red}$ agrees with the projection $\overline{N}_{\beta}\rightarrow \h_{2,\beta}^{red}$ on $\overline{N}_{\beta}$. 
\begin{lemma}\label{stratification}
The subset $\Theta_{\overline{N_{\beta}}}\subset \mathfrak{X}_2|_{N_{\beta}'}$ is closed. The subsets
\begin{equation}
\Theta_{N_{\beta}}-\Theta^{red}_{\beta}, \ \ \ \Theta_{\partial\overline{N}_{\beta}}, \ \ \ \Theta^{red}_{\beta}-\Theta^{red,sing}_{\beta},\ \ \ \Theta^{red,sing}_{\beta}
\end{equation}
of $\Theta_{\overline{N}_{\beta}}$ are locally closed in $\mathfrak{X}_2|_{N_{\beta}'}$ and form strata which endow $\Theta_{\overline{N}_{\beta}}$ with the structure of a Whitney stratified subset of $\mathfrak{X}_2|_{N_{\beta}'}$. With respect to this stratification, the projection  $\Theta_{\overline{N}_{\beta}}\rightarrow \h_{2,\beta}^{red}$ is a proper stratified submersion.
\end{lemma}
\begin{proof}
It follows from continuity of the projection $\mathfrak{X}_2|_{N_{\beta}'}\rightarrow N_{\beta}'$ that $\Theta_{\overline{N}_{\beta}}$ is closed in $\mathfrak{X}_2|_{N_{\beta}'}$. The subset $\Theta_{N_{\beta}}-\Theta^{red}_{\beta}$ of $\mathfrak{X}_2|_{N'_{\beta}}$ is open while $\Theta_{\partial\overline{N}_{\beta}}$ is closed. Thus they are both locally closed in $\mathfrak{X}_2|_{N'_{\beta}}$. Since $\Theta^{red,sing}_{\beta}$ is closed in $\mathfrak{X}_2|_{N'_{\beta}}$ it is locally closed in $\mathfrak{X}_2|_{N'_{\beta}}$, as is the open subset $\Theta^{red}_{\beta}-\Theta^{red,sing}_{\beta}$ of the closed subset $\Theta_{\beta}^{red}$. Since $\Theta^{red}_{\beta}\subset \Theta_{N_{\beta}}$ is a divisor with simple normal crossings and $\Theta_{\overline{N}_{\beta}}$ is a manifold with boundary $\Theta_{\partial\overline{N}_{\beta}}$, these locally closed subsets are the strata of a Whitney stratification of $\Theta_{\overline{N}_{\beta}}$ by \cite[pp.36-37]{goresky1988stratified} and the fact that an arrangement of linear subspaces in $\RR^n$ induces a stratification of $\RR^n$ by its flats \cite[p.236]{goresky1988stratified}.

 The projection is proper because both of the maps $\Theta_{\overline{N}_{\beta}}\rightarrow \overline{N_{\beta}}$ and $\overline{N}_{\beta}\rightarrow \h_{2,\beta}^{red}$ are. The final statement follows from the fact that $\Theta\rightarrow \h_2$ is a stratified submersion, the fact that $\Theta_{\partial N_{\beta}}\rightarrow \partial\overline{N}_{\beta}$ is a submersion, and the fact that the bundle projections $\partial\overline{N}_{\beta}, N_{\beta}\rightarrow \h_{2,\beta}^{red}$ are submersions.
\end{proof}
Observe that the fiber of $\Theta_{\overline{N}_{\beta}}\rightarrow \h_{2,\beta}^{red}$ over $p_{\beta}$ is $\Theta_{\overline{\Delta}_{\beta}}$. This receives an induced stratification from the stratification on $\Theta_{\overline{N}_{\beta}}$; its strata are the intersections of $\Theta_{\overline{\Delta}_{\beta}}$ with the strata of $\Theta_{\overline{N}_{\beta}}$. 

\begin{proposition}\label{prop31}
There is a commutative diagram  
\begin{equation}
\xymatrix{
\Theta_{\overline{N}_{\beta}}\ar[rd]\ar[rr]^{\cong\ \ \ \ }& & \h_{2,\beta}^{red}\times \Theta_{\overline{\Delta}_{\beta}}\ar[ld]\\
 &\h_{2,\beta}^{red} &
}
\end{equation}
where the top horizontal map is a stratum-preserving homeomorphism and the diagonal maps are the natural projections.
\end{proposition}
\begin{proof}
By Lemma \ref{stratification}, we may apply Thom's First Isotopy Lemma to the projection map $\Theta_{\overline{N}_{\beta}}\rightarrow \h_{2,\beta}^{red}$. This gives the required result.
\end{proof}
\begin{corollary}
There is a homeomorphism of pairs 
$$(\Theta_{\overline{N}_{\beta}}, \Theta_{\overline{N}_{\beta}} - \Theta^{red}_{\beta})\xrightarrow{\cong} \h_{2,\beta}^{red}\times (\Theta_{\overline{\Delta}_{\beta}}, \Theta_{\overline{\Delta}^*_{\beta}})$$
which commutes with projection to $\h_{2,\beta}^{red}$. This lifts to a homeomorphism 
$$(\Thetatilde_{\overline{N}_{\beta}}, \Thetatilde_{\overline{N}_{\beta}} - \Thetatilde^{red}_{\beta})\xrightarrow{\cong} \h_{2,\beta}^{red}\times (\Thetatilde_{\overline{\Delta}_{\beta}}, \Thetatilde_{\overline{\Delta}^*_{\beta}})$$
which commutes with projection to $\h_{2,\beta}^{red}$. Equivalently, there is a homeomorphism 
\begin{equation}
(\chat_{\overline{N}_{\beta}},\chat_{\overline{N}_{\beta}} - \chat^{red}_{\beta})\cong \h_{2,\beta}^{red} \times (\chat_{\overline{\Delta}_{\beta}},\chat_{\overline{\Delta}^*_{\beta}})
\end{equation}
commuting with projection to $\h_{2,\beta}^{red}$.
\end{corollary}

We now derive a few consequences of Proposition \ref{prop31}.
\begin{lemma}
The total space of the family $\chat_{\overline{N}_{\beta}}$ is homotopy equivalent to a 1-complex.
\end{lemma}
\begin{proof}
We will first show that the total space of the family $\chat_{\overline{\Delta}_{\beta}}\rightarrow \overline{\Delta}_{\beta}$ is homotopy equivalent to a 1-dimensional CW-complex. This will then be combined with Corollary 23 with Lemma 24.

We may assume that $\overline{\Delta}_{\beta}$ is sufficiently small that $\cC_{\overline{\Delta}_{\beta}}$ is homotopy equivalent to its central fiber $\C_0$, which is topologically the wedge $T^2\vee T^2$ of two 2-tori. Consequently, it is an aspherical space with fundamental group $\Gamma = \pi_1(T^2\vee T^2,*)$. 
Then $\chat_{\Delta_{\beta}}$ is an $H$-covering of $\C_{\overline{\Delta}_{\beta}}$ and so is an aspherical space with fundamental group isomorphic to the commutator subgroup $[\Gamma,\Gamma]$. On the other hand, the central fiber $\chat_0$ of $\C_0$ in $\chat_{\overline{\Delta}_{\beta}}$ is another such space and so there is a homotopy equivalence $\chat_0\simeq \chat_{\overline{\Delta}_{\beta}}$. But $\chat_0$ is a 1-dimensional Stein space and so, by a theorem of Hamm \cite{hamm1983homotopietyp} it has CW dimension 1.
\end{proof}
\begin{lemma}
The inclusion $\chat_{\partial \overline{\Delta}_{\beta}}\hookrightarrow \chat_{\overline{N}_{\beta}}-\chat^{red}_{\beta}$ is a homotopy equivalence. Thus $\chat_{\overline{N}_{\beta}}-\chat^{red}_{\beta}$ has the homotopy type of a CW complex of dimension at most 2.
\end{lemma}
\begin{proof}
By Corollary 23, the inclusion $\chat_{\overline{\Delta}_{\beta}^*}\hookrightarrow \chat_{\overline{N}_{\beta}}-\chat^{red}_{\beta}$ is a homotopy equivalence. Since the inclusion $\chat_{\partial\overline{\Delta}_{\beta}}\hookrightarrow \chat_{\overline{\Delta}_{\beta}^*}$ is also a homotopy equivalence, the result follows after we note that the fundamental group $\Gamma$ of $\chat_{\partial \overline{\Delta}_{\beta}}$ is given by an extension
\begin{equation}
1\rightarrow \pi'\rightarrow \Gamma\rightarrow \ZZ\rightarrow 1
\end{equation}
and therefore has geometric dimension at most 2 (see \cite[p.188]{brown1982cohomology}).
\end{proof}
\begin{remark}
\emph{
Actually, $\chat_{\overline{N}_{\beta}}-\chat^{red}_{\beta}$ has geometric dimension 2. By the Hochschild-Serre spectral sequence, there is an isomorphism $H_2(\Gamma, \ZZ)\cong H_1(\ZZ, H_1(\pi'))$. It is readily verified that this latter group is non-zero. 
}
\end{remark}
Equipped with these results, let us now return to our analysis of the long exact sequence of the terms $H_3(\chat_{N_{\beta}},\chat_{N_{\beta}}-\chat^{red}_{\beta})$. An application of Lemmas 24 and 25 reveals that the long exact sequence of homology for the pair $(\chat_{N_{\beta}},\chat_{N_{\beta}}-\chat^{red}_{\beta})$ has a segment
\begin{equation}
0\rightarrow H_3(\chat_{N_{\beta}},\chat_{N_{\beta}}-\chat^{red}_{\beta})\xrightarrow{\partial} H_2(\chat_{N_{\beta}}-\chat^{red}_{\beta})\rightarrow 0.
\end{equation}
Piecing all of these isomorphisms together, we now have a commutative diagram
\begin{equation}
\xymatrix{
H_3(\chat,\chat-\chat^{red})\ar[r]^{\partial}& H_2(\chat-\chat^{red})\\
\bigoplus_{\beta}H_3(\chat_{\overline{\Delta}_{\beta}},\chat_{\overline{\Delta}_{\beta}}- \chat^{red}_{\beta} )\ar[r]^{\ \ \ \ \ \ \ \ \partial}_{ \ \ \ \ \ \ \cong}\ar[u]^{\cong} & \bigoplus_{\beta}H_2(\chat_{\partial\overline{\Delta}_{\beta}})\ar@{^{(}->}[u]
}
\end{equation}

This implies that there is an exact sequence
\begin{equation}
0\rightarrow \cF'\rightarrow H_2(\chat)\rightarrow H_2(\chat,\chat-\chat^{red})
\end{equation}
where $\cF' \cong H_2(\chat-\chat^{red}) /H_2(\chat_{\overline{\Delta}_{\beta}})$. Since $H_2(\chat,\chat-\chat^{red})$ is a free abelian group (it is freely generated by small discs transverse to the components of $\chat^{red}$), the image of $H_2(\chat)$ in $H_2(\chat,\chat-\chat^{red})$ is also a free abelian group. Thus $H_2(\chat)$ is an extension of a free abelian group by a free abelian group and is therefore also free abelian.

Putting all of this together, we have deduced 
\begin{proposition}\label{prop31}
The boundary map $\partial: H_3(\chat,\chat-\chat^{red})\rightarrow H_2(\chat-\chat^{red})$ is injective, and its cokernel is a free abelian group. Consequently, the group $H_3(\chat)$ vanishes and $H_2(\chat)$ is a free abelian group.
\end{proposition}

We are now in a position to deduce the homotopy type of $\chat$. Our strategy is to apply the homological form of Whitehead's Theorem (see \cite[p.367]{hatcher606algebraic}).
\begin{proposition}
There is a homotopy equivalence $\displaystyle \bigvee_{I} S^2\rightarrow \chat$ for some index set $I$ whose cardinality matches the rank of $H_2(\chat)$.
\end{proposition}
\begin{proof}
Since $\chat$ is simply-connected, the Hurewicz Theorem gives an isomorphism $\pi_2(\chat)\rightarrow H_2(\chat)$. Thus each generator of $H_2(\chat)$ is represented by a continuous map $S^2\rightarrow \chat$. By taking the wedge of all these maps, we obtain a continuous map 
\begin{equation}
f:\bigvee_{I} S^2\rightarrow \chat
\end{equation}
which induces isomorphisms $f_*:H_{\bullet}\left(\bigvee S^2\right)\xrightarrow{\cong} H_{\bullet}(\chat)$ on all homology groups. Here $I$ is an index set whose cardinality matches the rank of $H_2(\chat)$. Since $\bigvee S^2$ and $\chat$ are both simply-connected and each admits the structure of a CW complex, Whitehead's Theorem implies that $f$ is a homotopy equivalence.
\end{proof}

\subsection{\texorpdfstring{The Rank of $H_2(\chat)$}{The rank of H2(chat,ZZ)}}
We will now show that the rank of $H_2(\chat)$ is infinite. The proof is not constructive, but in the next section we give a way of producing reasonably explicit elements. 
\\\\\indent
We begin with the following observation. Since $\chat$ covers $\C$ with covering group $H$, there is a Hurewicz fibration
$$\chat\rightarrow \C\rightarrow K(H,1)$$
and therefore a spectral sequence
\begin{equation}
E^2_{s,t} = H_s(K(H,1), H_t(\chat))\implies H_{s+t}(\C)
\end{equation}
converging to the homology of $\C$. 

By the computation of the homology of $\chat$, the $E^2$ page of this spectral sequence is rather sparse. It has only two nonzero rows, the relevant parts of which are
\begin{equation*}
\begin{tabular}{l|r}
 & \xymatrix{
   0&0&0&0 \\
  H_0(H, H_2(\chat))&  H_1(H, H_2(\chat))& H_2(H, H_2(\chat)) & H_3(H, H_2(\chat))   \\
  0&0&0&0  \\
  \ZZ & H &\Lambda^2 H & \Lambda^3H 
  }\\
\hline\\
\end{tabular}
\end{equation*}
where $H_s(H, H_2(\chat)):=H_s(K(H,1), H_2(\chat))$. This spectral sequence degenerates at the $E^4$ page. Consequently, there is an isomorphism
\begin{equation}
H_2(\C,\QQ)\cong \Lambda^2H_{\QQ} \oplus H_0(H, H_2(\chat,\QQ))/d^2_{3,0}(\Lambda^3H)\otimes {\QQ}.
\end{equation}
where $H_{\QQ} = H\otimes_{\ZZ}\QQ$. Since $\Lambda^j H_{\QQ}$ is finite-dimensional, the group $H_2(\chat)$ has infinite rank if $H_2(\C)$ does. We will demonstrate that $H_2(\C)$ does not have finite rank by showing that it surjects onto a free abelian group of infinite rank.
\\\\\indent 
 
The long exact sequence of integral homology for the pair $(\C,\C-\C^{red})$ has a segment
fitting into a commutative diagram
\begin{equation}
\xymatrix{
& &H_1(\C_{\Omega_0})\ar@{^{(}->}[d] & & \\
H_2(\C) \ar[r]& H_2(\C, \C-\C^{red}) \ar[r]\ar[d]& H_1(\C-\C^{red})\ar[r]\ar@{^{}->>}[d] & H_1(\C) \ar[r] & 0 \\
& H_2(\h_2,\h_2-\h_2^{red}) \ar[r]^{\ \ \cong} & H_1(\h_2-\h_2^{red}) & &
}
\end{equation}
where $C_{\Omega_0}$ is any smooth fiber. With the aid of this sequence, one proves
\begin{lemma}
The sequence
\begin{equation}
H_2(\C)\rightarrow H_2(\C,\C-\C^{red})\rightarrow H_2(\h_2,\h_2-\h_2^{red})\rightarrow 0
\end{equation}
is exact.
\end{lemma}
\begin{proof}
This is a diagram chase. The crucial observation is that the composition
\begin{equation}
H_1(\C_{\Omega_0})\rightarrow H_1(\C-\C^{red})\rightarrow H_1(\C)
\end{equation}
is an isomorphism.
\end{proof}

Our aim is to show that the projection $H_2(\C,\C-\C^{red})\rightarrow H_2(\h_2,\h_2-\h_2^{red})$ has a kernel of infinite rank. 
 We will make use of an integral basis for $H_2(\C,\C-\C^{red})$. To construct it, arbitrarily label the two components of $\C^{red}_{\beta}$ with a ``1" and a ``2", respectively; denote these components by $D_{\beta,1}$ and $D_{\beta,2}$. Then it is possible to choose a sufficiently small disc $\Delta_{\beta}$ which is transverse to $\h_{2,\beta}^{red}$ at a single point such that there are two holomorphic sections $s_{\beta,1},s_{\beta,2}: \Delta_{\beta}\rightarrow \C_{\Delta_{\beta}}$ with the property that $s_{\beta,j}$ intersects $D_{\beta,j}$ transversely in a single point. Regarded as elements of $H_2(\C,\C-\C^{red})$, the sections $\{s_{\beta,j}\}_{\beta,j}$ form an integral basis for $H_2(\C,\C-\C^{red})$.
 
 At the same time, the discs $\Delta_{\beta}$ form a basis for $H_2(\h_2,\h_2-\h_2^{red})$. It is clear that the image of $s_{\beta,j}$ in $H_2(\h_2,\h_2-\h_2^{red})$ equal to $\Delta_{\beta}$ for both choices of $j$. Thus the kernel of the projection $H_2(\C-\C^{red})\rightarrow H_2(\h_2,\h_2-\h_2^{red})$ is freely generated by the differences $\{s_{\beta,1}-s_{\beta,2}\}_{\beta}$. This yields
\begin{proposition}\label{prop34}
The group $H_2(\C)$ has infinite rank and, as a consequence, the group $H_2(\chat)$ has infinite rank.
\end{proposition}

\section{\texorpdfstring{Constructing Elements of $H_2(\chat)$}{Constructing Elements of H_2chat}}

It follows from Mess's description \cite{mess1992torelli} of the monodromy of $\C-\C^{red}$ that any oriented SSCC $\ell$ on a smooth fiber $\C_0$ of $\C-\C^{red}$ is freely homotopic in $\C-\C^{red}$ to an oriented vanishing cycle $v_{\beta}$ in one of the subfamilies $\C_{\Delta_{\beta}}$. The precise value of the index $\beta$ depends on the homology splitting of $\C_0$ induced by $\ell$. 

Lifting to $\chat-\chat^{red}$, this implies that any lift $\tilde\ell$ of $\ell$ to $\chat-\chat^{red}$ is freely homotopic to a lift $\tilde v_{\beta}$ of the vanishing cycle $v_{\beta}$. Assuming that $\Delta_{\beta}$ is sufficiently small, one shows that $\tilde v_{\beta}$ is homologous in $\chat-\chat^{red}$ to a 1-cycle of the form $\pm (\partial \tilde s_{\beta,1}-\partial \tilde s_{\beta,2})$, where $\tilde s_{\beta,j}$ denotes the lift of a section $s_{\beta,j}:\Delta_{\beta}\rightarrow \C_{\Delta_{\beta}}$ to $\chat_{\Delta_{\beta}}$ and $\partial \tilde s_{\beta,j}$ denotes its restriction to the boundary $\partial \Delta_{\beta}$ of $\Delta_{\beta}$. It follows that the image of the homology class of $\tilde \ell$ in $H_1(\chat-\chat^{red})$ is equal to the image of $\pm (\tilde s_{\beta,1}-\tilde s_{\beta,2})$ under the boundary map $\partial: H_2(\chat, \chat-\chat^{red})\rightarrow H_1(\chat-\chat^{red})$. Certainly $\tilde s_{\beta,1}-\tilde s_{\beta,2}$ is a non-zero. 

Now suppose we have a collection of pointed, oriented SSCCs $\{\ell_k\}_{k=1}^n$ on $\C_0$, all with the same basepoint $x_0$. Further suppose the $\ell_k$ distinct homology splittings on $\C_0$ and that the product $\ell_1\cdots \ell_n$ is trivial in $\pi_1(\C_0,x_0)$. Let $\{\tilde \ell_k\}$ be the set of of pointed lifts of the $\ell_k$ to $\chat_0\subset \chat-\chat^{red}$  with basepoint $\tilde x_0\in \chat_0$. Then there exists a non-zero element
\begin{equation}
S = \sum_k \pm\left(\tilde s_{\beta(k),1}-\tilde s_{\beta(k),2}\right)\in H_2(\chat,\chat-\chat^{red})
\end{equation}
with the property that $\partial S = \displaystyle \sum_k [\tilde \ell_k] = 0\in H_1(\chat-\chat^{red})$. Here $\beta(k)$ indexes the component of $\h_2^{red}$ corresponding to the homology splitting induced by $\ell_k$, and $[\tilde \ell_k]$ denotes the homology class of $\tilde \ell_k$ in $\chat-\chat^{red}$.
\\\\\indent 
Very shortly, we will argue that there exist non-empty collections $\{\ell_k\}$ of such SSCCs. In order to do so, we work with the standard reference surface $S_2$. Let $\{a_1,b_1,a_2,b_2\}$ denote the generating set for $\pi_1(S_2, *)$ pictured below. The crucial tool here is the Hall-Witt identity (see \cite{magnus2004combinatorial}, p. 290), which states that, for any elements $x,y,z$ in a group $G$, 
\begin{equation}
[x,yz] = [x,y]\cdot [x,z]^{y^{-1}} \ \ \ \ \ \text{or alternatively}\ \ \ \ \ [xy,z] = [y,z]^{x^{-1}}[x,z]
\end{equation}
where $[a,b] = aba^{-1}b^{-1}$ and $a^b = b^{-1}ab$.
\begin{figure}[h] 
\centering{
\includegraphics[scale=0.5]{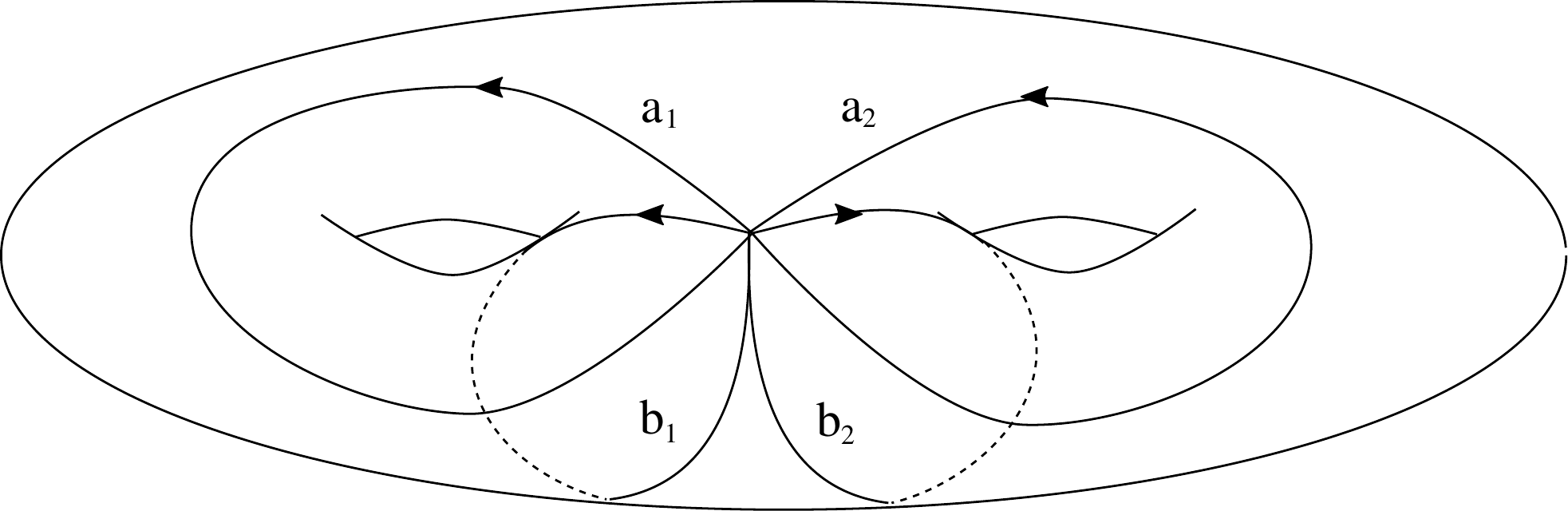}}
\caption{Curves used in the proof of Lemma \ref{lemma34}}
\end{figure} 
\begin{lemma}\label{lemma34}
The commutators $[a_1,b_1a_2], [b_1^{-1},a_2a_1]^{b_1^{-1}a_1^{-1}b_1^{-1}}, [a_1,b_1]$ and $[a_1b_1,a_2a_1]^{b_1^{-1}}$  are represented by oriented SSCCs on $S_2$. They induce homology splittings that do not all coincide, and there is a multiplicative relation between them of the form
\begin{equation}
[a_1,b_1a_2] = [a_1,b_1][b_1^{-1},a_2a_1]^{b_1^{-1}a_1^{-1}b_1^{-1}} [a_1b_1,a_2a_1]^{b_1^{-1}}.
\end{equation}
\end{lemma}
\begin{proof}
The proof that the four loops given above are SSCCs is an easy, if somewhat tedious, verification. The corresponding homology splittings are given in the table below.

\vspace{.01in}
\begin{figure}[h]
\caption{SSCC's and the induced homology splittings.}
\begin{center}
    \begin{tabular}{| l | l | l | l |}
    \hline
    SSCC & Homology Splitting  \\ \hline
    $[a_1,b_1a_2]$  & $\ZZ\langle a_1, b_1+a_2\rangle\oplus \ZZ\langle a_2, b_2+a_1\rangle$ \\ \hline
    $[b_1^{-1},a_2a_1]^{b_1^{-1}a_1^{-1}b_1^{-1}}$   & $\ZZ\langle -b_1, a_1+a_2\rangle\oplus \ZZ\langle a_2, b_2-b_1\rangle$ \\ \hline
    $[a_1,b_1]$     & $\ZZ\langle a_1, b_1\rangle\oplus \ZZ\langle a_2, b_2\rangle$ \\ \hline
    $[a_1b_1,a_2a_1]^{b_1^{-1}}$   & $\ZZ\langle a_1+b_1, a_1+a_2\rangle\oplus \ZZ\langle a_2, b_2-a_1-b_1\rangle$  \\ \hline
        \end{tabular}
\end{center}
\end{figure}
It is easily seen that the homology splittingss appearing in Figure 2 are all distinct. Finally we verify the commutator identity in the statement of the lemma. It is a striaghtforward application of the Hall-Witt identities. We remark that the proof applies equally well in any group. Start with the separating simple closed curve $[a,bc]$ and expand using Hall-Witt: $[a_1,b_1a_2] = [a_1,b_1]\cdot [a_1,a_2]^{b_1^{-1}}$.
Using the identity $[x,y] = [x,yx]$ we may write
$$[a_1,a_2] = [a_1,a_2a_1]= [a_1b_1b_1^{-1},a_2a_1] = [b_1^{-1},a_2a_1]^{b_1^{-1}a_1^{-1}}\cdot [a_1b_1,a_2a_1]$$
so that finally we have\ \ $[a_1,b_1a_2] =  [a_1,b_1]\cdot [b_1^{-1},a_2a_1]^{b_1^{-1}a_1^{-1}b_1^{-1}}\cdot [a_1b_1,a_2a_1]^{b_1^{-1}}$.
\end{proof}

\vspace{.15in}
\noindent Kevin Kordek\\
Department of Mathematics\\
Mailstop 3368\\
Texas A$\&$M University\\
College Station, TX 77843-3368\\
 E-mail: \sf{kordek@math.tamu.edu}

\end{document}